\documentclass{article}
\usepackage{graphicx, amsmath, amssymb, amsthm}
\usepackage{color}

\newtheorem{theorem}{Theorem}[section]

\newtheorem{corollary}[theorem]{Corollary}

\newtheorem{lemma}[theorem]{Lemma}

\newtheorem{proposition}[theorem]{Proposition}

\title{VC-dimension of subsets of Hamming graphs}
\author{Christopher Housholder, Layna Mangiapanello, and Steven Senger}
\date{\today}

\begin{document}

\maketitle

\begin{abstract}
    Following recent work on the VC-dimension of subsets of various pseudorandom graphs, we study the VC-dimension of Hamming graphs, which have proved somewhat resistant to the standard techniques in the literature. Our methods are elementary, and agree with or improve upon previously known results. In particular, for $H(2,q)$ we show tight bounds on the size of a subset of vertices to guarantee VC-dimension 2 or 3. We also prove an assortment of results for other parameters, with many of these being tight as well.
\end{abstract}

\section{Introduction}
The notion of Vapnik-Chervonenkis or VC-dimension, from \cite{VC}, has been used for a variety of applications, ranging from machine learning (see \cite{BS}) to combinatorial applications (see \cite{AlonSpencer}). Recently, there has been interest in estimating the VC-dimension of subsets of the vertex sets of various graphs, as in \cite{FIMW, IMS}. Many of the methods in this pursuit rely on the pseudorandomness of the underlying graph. While the Hamming graph is pseudorandom, it has been less susceptible to these techniques, as shown in \cite{PSTT}. Building on this work, we use elementary methods to prove estimates (many tight) on the VC-dimension of subsets of Hamming graphs.

For natural numbers $d$ and $q,$ one way to describe the Hamming graph $H(d,q)$ is with vertices $\mathbb Z_q^d,$ that is $d$-tuples of integers in the interval $[0,q),$ where vertices are adjacent if they differ in exactly one coordinate. More generally, for an integer $t,$ the Hamming graph $H(d,q,t)$ has the same vertex set, but vertices are adjacent when they differ in exactly $t$ coordinates.

Given a collection of subsets $\mathcal F$ of some ambient set $X,$ and some fixed subset $W\subseteq X,$ we say $\mathcal F$ {\it shatters} $W$ if for every $S$ in the power set of $W,$ there is a subset $F_S\in \mathcal F$ so that $F_S\cap W = S.$ That is, every possible subset of $W$ is in the intersection of $W$ and some member of the family $\mathcal F.$ The {\it Vapnik-Chervonenkis-dimension} or {\it VC-dimension} of the pair $(X,\mathcal F)$ is the size of the largest subset $W\subset X$ that can be shattered by $\mathcal F.$ We often call such a subset a {\it shattering set}.

For example, if $X$ is the plane, and $\mathcal F$ is the set of lines, the VC-dimension of $(X,\mathcal F)$ is 2, because for any pair of points, $x,y\in\mathbb R^2,$ we can find lines that contain both, neither, or exactly one of $x$ and $y.$ However, we could never find a collection of lines that shatters any set of three points $x,y,z\in \mathbb R^2,$ because if the three points do not lie on a line, there is no line that contains all three, and if they do lie on a line, there is no line that contains exactly two of them.

One type of problem that has been around for a while, but gained more attention in recent years, is the estimation of the VC-dimension of subsets of certain graphs. Recall that the {\it neighborhood} of a vertex $x$, denoted $n(x),$ is the set of vertices adjacent to $x.$ In particular, given a graph $G(V,E),$ where $V$ is the set of vertices, $n(V)$ is the set of neighborhoods of vertices, and $E$ is the set of edges, there is much interest in determining the VC-dimension of the pair $(V, n(V)).$ When context is clear, we refer to this as the VC-dimension of the graph $G.$ See \cite{ABC, FIMW, IMS, ABCIJLMMMRV} and the references included therein for more information.

In \cite{ABM}, the authors show that for most choices of $d$ and $q,$ the VC-dimension of $H(d,q)$ is 3. In \cite{PSTT}, the authors gave a suite of results about the VC-dimension of induced subgraphs of a family of graphs called pseudorandom graphs (See \cite{H80, KS} for more on these and the techniques employed in their study.). However, they noted that while the Hamming graph is pseudorandom (see \cite{BCIM} and the references contained therein), it was resistant to the standard techniques, in that they yielded comparatively weak results. This was demonstrated by proving stronger results with elementary techniques.

The main goal of this note is to improve upon and generalize some the initial results on Hamming graphs from \cite{PSTT}. The general strategy for each proof will be to show that some type of point configuration must (or respectively, must not) be present in some subset of the vertices of the appropriate Hamming graph. The following result first appeared in \cite{PSTT}, but only for $d=2,$ and the proof was a longer induction argument, based on finding a more intricate point configuration, than the proof we give here.

\begin{proposition}\label{prop22}
    Given a natural number $q\geq 3,$ and subset $U$ of the vertices of $H(2,q)$ of size $|U|\geq 2q^{d-1}+1$, the VC-dimension of $(U,n(U))$ is at least 2.
\end{proposition}

We now focus on the $d=2$ case, where we prove that Proposition \ref{prop22} is tight in general, but that the bound can be improved if $q$ is odd.

\begin{theorem}\label{thm22odd}
    Given an odd number $q\geq 3,$ and subset $U$ of the vertices of $H(2,q)$ of size $|U|\geq 2q$, the VC-dimension of $(U,n(U))$ is at least 2. Moreover, this bound and the bound in Proposition \ref{prop22} for $d=2$ are tight.
\end{theorem}

Of course, if one has a larger subset of $H(2,q),$ we can guarantee the maximum possible VC-dimension. We also demonstrate that this bound is tight.

\begin{theorem}\label{thm23}
    Given a natural number $q\geq 4,$ and subset $U$ of the vertices of $H(2,q)$ of size $|U|\geq 3q+1$, the VC-dimension of $(U,n(U))$ is 3. Moreover this bound is tight.
\end{theorem}

We pause to note that by pigeonholing, Theorem \ref{thm23} yields nontrivial results for higher ambient dimensions. We record these in the following corollary.

\begin{corollary}\label{phpCor}
    Given natural numbers $d, q\geq 4,$ and a subset of $U$ of the vertices of $H(d,q)$ of size $|U|\geq 3q^{d-1}+1,$ then the VC-dimension of $(U, n(U))$ is 3.
\end{corollary}
\begin{proof}
The graph $H(d,q)$ has $q^d-2$ isomorphic copies of $H(2,q),$ so by pigeonholing, there must be an isomorphic copy of $H(2,q)$ with at least $3q+1$ points on it, and we apply Theorem \ref{thm23} there.
\end{proof}

While we currently do not know how sharp the results about VC-dimension 2 are in higher ambient dimensions are, we have an explicit construction that shows that we at least have the correct exponent of $q.$

\begin{proposition}\label{sharp3}
    Given an even natural number $q\geq 4,$ there exists a subset $U$ of the vertices of $H(3,q)$ of size $|U|= \frac{5}{4}q^{2}$ so that the VC-dimension of $(U,n(U))$ is 1.
\end{proposition}

While Proposition does require $q$ to be even, one can prove a similar result for $q$ odd by embedding the construction above into $H(d,q-1)$ and adding one more point with all coordinates $(q-1).$ However, the constant drops off when $d$ grows large, so we also give the following result, which comes from a simpler construction that has no parity constraint on $q$ when $d\geq 4$.

\begin{proposition}\label{sharp4+}
    Given natural numbers $d, q\geq 2,$ there exists a subset $U$ of the vertices of $H(d,q)$ of size $|U|=q^{d-1}$ so that the VC-dimension of $(U,n(U))$ is 1.
\end{proposition}

We then show that Corollary \ref{phpCor} is sharp. That is, we have completely characterized which subsets of the vertices of $H(d,q)$ have VC-dimension 3 in terms of size.

\begin{proposition}\label{3sharp4+}
    Given natural numbers $d, q\geq 2,$ there exists a subset $U$ of the vertices of $H(3,q)$ of size $|U|=3q^{d-1}$ so that the VC-dimension of $(U,n(U))$ is at most 2.
\end{proposition}

Finally, we offer some initial results about subsets of $H(d,q,t)$ where $t\neq 1.$ In particular, we prove the following tight result in the case that $t=d=2.$

\begin{theorem}\label{thm222}
    Given a natural number $q\geq 3,$ and subset $U$ of the vertices of $H(2,q,2)$ of size $|U|\geq 2q$, the VC-dimension of $(U,n(U))$ is at least 2. Moreover, this bound is tight.
\end{theorem}

In the next section, we prove the main results about when subsets of $H(2,q)$ have VC-dimension at least two. The following section does the same with VC-dimension three. After that, we produce some constructions to demonstrate how tight or loose our various bounds are. The final section contains some preliminary results on the behavior of $H(d,q,t)$ where $t\neq 1.$

\section{VC-dimension two}
We begin by proving Proposition \ref{prop22}, which holds for $d\geq 2$, then restrict to $d=2$ to prove Theorem \ref{thm22odd}. Here and below, we say `line' to mean a rectilinear line (parallel to a basis vector). So a plane is a set of points where exactly one coordinate takes all possible values, and the rest remain fixed. For ease of exposition, we occasionally say `horizontal' or `row' to mean parallel with the first basis vector, and `vertical' or `column' to mean parallel with the second basis vector. Similarly, we will use the term `plane' to refer to rectilinear planes, or a set of points with exactly two coordinates taking all possible values, and the rest remaining fixed. Moreover, as the vertices of a given Hamming graph may be viewed as points with relationships organized by sets of lines, we will often conflate the terms `point' and `vertex' depending on context. To keep track of vertices using coordinates, we often identify the vertex set of $H(d,q)$ with $\mathbb Z_q^d.$
\subsection{Arguments for $d\geq 2$}
The general strategy will be to look for a set of three points on the same vertical or horizontal line, as well as a point off of that line, all within our subset $U$. We prove that a subset $U$ of the vertices of $H(2,q)$ with this configuration will guarantee that $(U,n(U))$ has VC-dimension at least two.
\subsubsection{Three points on a line}
\begin{lemma}\label{lemma22}
    Given a natural number $q\geq 3,$ and subset $U$ of the vertices of $H(2,q)$ if $U$ contains three points on the same line $L$, as well as a fourth point not on $L,$ then the VC-dimension of $(U,n(U))$ is at least 2.
\end{lemma}
\begin{proof}
    To see that this suffices, suppose the points $x,y,$ and $z$ lie on a line $L$, and $u_0$ is a point not on $L$. At most one of $x,y,$ or $z$ could also be on a line $L'\neq L$ with $u_0.$ Without loss of generality, suppose $x$ and $y$ do not lie on $L'.$ Now we set about showing that the set $W:=\{x,y\}$ is shattered by the neighborhoods of $u_0,x,y,$ and $z.$ Specifically, notice that $z$ differs by both $x$ and $y$ in exactly one coordinate, so it is adjacent to both $x$ and $y.$ Therefore $n(z),$ the neighborhood of $z,$ intersects with $W$ in both points. Similarly, $n(x)\cap W = \{y\},$ and $n(y)\cap W = \{x\}.$ Finally, $n(w)\cap W=\varnothing,$ because $w$ is not on a line with either of $x$ or $y.$
\end{proof}

\begin{figure}
\centering
\includegraphics[scale=1.2]{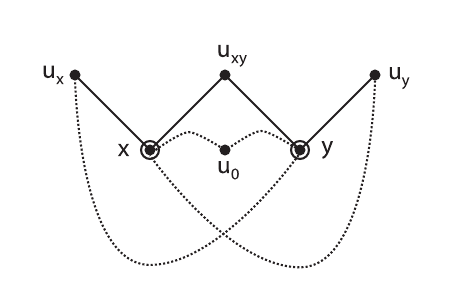}
\caption{The shattering set is $W=\{x,y\}.$ When two vertices are adjacent, an edge is indicated with a solid line. When two vertices may not be adjacent, this is indicated with a dotted curve. In the proof of Lemma \ref{lemma22}, $z$ served as $u{xy}.$}
\label{figure4}
\end{figure}

\subsubsection{Proof of Proposition \ref{prop22}} Notice that by the fact that there are exactly $q^{d-1}$ horizontal lines, the size condition $|U|\geq 2q^{d-1}+1,$ and the pigeonhole principle, there must be a horizontal line $L$ with at least three points. No horizontal line can have more than $q$ points on it, so we must have at least one point off of any given horizontal line, including $L.$ Therefore we have the required configuration of points to apply Lemma \ref{lemma22}, and we are done.
\subsection{Proof of Theorem \ref{thm22odd}}
We now focus our attention on results where the ambient dimension is $d=2.$ Theorem \ref{thm22odd} will follow by combining the following lemmas.
\subsubsection{Tightness}
We prove the sharpness of Proposition \ref{prop22} in general by construction for even $q$. The basic idea is to pack in $2\times 2$ boxes of points so that there are $2q$ points in total, but the VC-dimension of the subset is 1. We then provide a modified construction for the case that $q$ is odd.
\begin{lemma}\label{vc2tight}
Given a natural number $q,$ there exists a subset $U_1(q)$ of the vertices of $H(2,q)$ of size $|U_1(Q)|=2q-i,$ where $i\equiv q$ mod 2, so that the VC-dimension of $(U_1(q),n(U_1(q)))$ is 1.
\end{lemma}
\begin{proof}
To see this, define the set
\[B:=\{(0,0),(0,1),(1,0),(1,1)\}\]
and consider
\[U_1(q):=\bigcup_{i=0}^{q/2}\left(B+(2i,2i)\right),\]
where the addition denotes the standard Minkowski sum, or componentwise addition. We now show that any subset $W\subset U_1(q)$ of size 2 cannot be shattered. Pick any pair of points $\{x,y\}$ from $U_1(q).$ Either they are in the same translate of $B$ or not. If they are not in the same translate, then they cannot both be adjacent to any point in $U_1(q).$ Therefore, they must be in the same translate of $B$ if we are to have any hope of shattering them. If they are in the same translate of $B,$ then either they are adjacent or not. If they are adjacent, notice that there is no point from $U_1(q)$ adjacent to both of them. If they are not adjacent, then there is no point from $U_1(q)$ adjacent to exactly one of them. Therefore, the VC-dimension of $U_1(q)$ is strictly smaller than 2, and Proposition \ref{prop22} is sharp in general. Moreover, it is clear that any point $x\in U_1(q)$ is adjacent to some point $y\in U_1(q)$, and not adjacent to some other point $z\in U_1(q),$ so the VC-dimension is 1.

We can modify the construction above to show that the case that $q$ is odd is sharp when $|U|=2q-1$ by starting with the construction for $U_1(q-1)\subseteq \mathbb Z_{q-1}\times \mathbb Z_{q-1}.$ Notice that $|U_1(q-1)|=2q-2.$ Now we can consider this as a subset of $\mathbb Z_q\times \mathbb Z_q,$ and add in the point $(q-1,q-1)$ for a total of $2q-1$ points. More precisely, for $q$ odd, we define
\[U_1(q):=U_1(q-1)\cup\{(q-1,q-1)\}.\]
Arguing as we did in the case that $q$ was even (except that we cannot choose $x=(q-1,q-1)$), we see that this set also has VC-dimension 1.
\end{proof}
\subsubsection{The lower bound for odd $q$}
\begin{lemma}\label{nextOne}
    Given an odd number $q\geq 3,$ and a subset $U$ of the vertices of $H(2,q)$ of size $|U|\geq 2q,$ the VC-dimension of $(U,n(U))$ is at least 2.
\end{lemma}
\begin{proof}
    Since $q$ is odd, we can write it as $q=2p+1$ for some natural number $p$. Since Proposition \ref{prop22} already dealt with the case that $|U|\geq 2q+1,$ we can turn our attention to the case $|U|=2q.$ Suppose that we are given some subset $U$ of the vertices of $H(2,q)$ with $|U|= 2q.$ Now, we are done if we can find three points on a line, because by the size condition alone, no line can have more than $q$ points, so there must always be a fourth point off of any line with at least three, and we would then be finished by Lemma \ref{lemma22}. So the remainder of this proof will be in showing that if $U$ does not have three points on any line, then it still must have a set of size two that is shattered by $n(U).$

    Suppose that $U$ does not have three points on any line. By the size assumption, $|U|=2q,$ and the fact that there are exactly $q$ vertical lines, we see that there must be exactly two points on each vertical line. Similarly, every horizontal line must have exactly two points. Next, we will search for a different type of point configuration that will guarantee the existence of a shattering set in this particular case. Specifically, for some point $(a,c)$ that is not in $U,$ we claim that we can find a triple of points of the following form, which we call the {\it corner} of $(a,c)$:
    \[\{(a,d),(b,c),(b,d)\}.\]
    These are three points from $U$ that are the vertices of a rectangle whose fourth point is not in $U.$ To see that there must be a corner for some point $(a,c)\notin U,$ we use a greedy algorithm, processing points in $U$ until we find what we need. Notice that since every line has exactly two points, every point $(i,j)\in U$ is on the intersection of two lines that have other points of $U$, say $(i,j')$ and $(i',j),$ for some $i'\neq i,$ and $j'\neq j.$ Now, if $(i',j')\notin U,$ we are done. If $(i',j')\in U,$ we ignore that quadruple of points and the lines they populated and continue. We are safe in ignoring those lines, as there are exactly two points on each line, so that quadruple of points accounts for all points on each of the four lines in question (two vertical and two horizontal). We may be able to throw out quadruples like this for a while\footnote{Notice that in the case that $q$ was even, such quadruples could be exhaustive. This was precisely the case for each of the shifts of $B$ in the construction of the set $U_1(q)$ when $q$ was even.}, but the process of eliminating quadruples must eventually terminate with some points left over, because $|U|=2q=2(2p+1)=4p+2,$ which is not a multiple of 4. Therefore, there is some triple of points $(a,d),(b,c),(b,d)\in U$ that form a corner for a point $(a,c)\notin U.$

    To finish, we show that these will guarantee the existence of a two-element subset $W\subseteq U$ that can be shattered by $n(U).$ Specifically, choose
    \[W=\{(a,d),(b,c)\}.\]
    First we see that $(b,d)$ is adjacent to both. Next, notice that $(a,d)$ is on exactly two lines with exactly two points each. We already know that $(b,d)$ is on one of those lines, but we also know that the other point from $U$ on the other line cannot be $(a,d),$ as $(b,d)\notin U.$ So there must be some point $(a',d)\in U$ with $a'\neq a$ and $a'\neq b.$ That is, $(a',d)$ is adjacent to one of our points in $W$ but not the other. Similarly, there must be a point $(b,c')\in U$ with $c'\neq c$ and $c'\neq d.$ So the point $(b,c')$ is also adjacent to exactly one point in $W,$ but not the same point that $(a',d)$ was adjacent to. Finally, we know that each of the points in $W$ is adjacent to exactly two points, and we have explicitly dealt with five points total. However, since $q\geq 3,$ we know that $|U|=2q>5,$ meaning there must be some other point adjacent to neither of the points in $W,$ and we are done.
\end{proof}
\section{VC-dimension three}
Theorem \ref{thm23} will follow by combining the lemmas given here.
\subsection{Tightness}
\begin{lemma}\label{lemma23}
    Given a natural number $q\geq4$, and a subset of $U$ of the vertices of $H(2,q)$ if the VC-dimension of $(U,n(U))$ is 3, then there must be at least one row or column with four points on the same vertical or horizontal line $L$.
\end{lemma}
\begin{proof}
    We will proceed by way of contradiction. Suppose there is a subset $U$ of the vertices of $H(2,q)$ with no more than three points on any horizontal or vertical line, but so that $(U,n(U))$ has VC-dimension 3. That is, we will assume that there is a subset $W\subseteq U$ with $|W|=3,$ that is shattered by the points in $U.$ There are four possible configurations of points forming the shattering set which we address in four separate cases. We will see that the condition of no more than three points on any line will make it so that none of them can be shattering sets.

\begin{figure}
\centering
\includegraphics[scale=1.2]{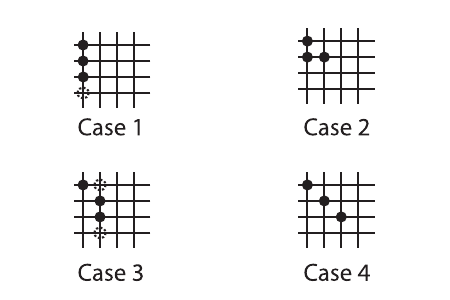}
\caption{The four cases, with the shattering points filled in, and other points of interest indicated with dotted lines.}
\label{figure1}
\end{figure}

    Case 1: All three points from the shattering set are on a line. 
    We will consider these points to be 
    \[\{(a,b),(a,c),(a,d)\}.\]
    With these three points and exactly three points on any one line, there does not exist a fourth point $(a,e)\in U$ and thus there is a contradiction since there is no point adjacent to all three of the points, so VC-dimension cannot be 3.

    Case 2: The three points from the shattering set form a corner.
    We will consider these points to be 
    \[\{(a,b), (c,b), (a,c)\}.\]
    Since it is VC-dimension 3, there must be a point adjacent to all three shattering points. This point does not exist and thus the VC-dimension cannot be 3.

    Case 3: Two points from the shattering set are on a line and the third is not on any line with the other two. Without loss of generality two will be on a vertical line.
    We will consider these points to be 
    \[\{(a,b), (c,d), (c,e)\}.\]
    Since $(U,n(U))$ has VC-dimension 3, the point $(c,b)$ must be in $U$ because it is the only point adjacent to all three other points. Since there are exactly three points in a line, there is no point in $U$ that is adjacent to both $(c,d)$ and $(c,e)$ but not $(a,b),$ which is a contradiction and thus the VC-dimension cannot be 3. 

    Case 4: If none of the three shattering points are on any line, then there can be no point adjacent to all three of them, and the VC-dimension is $<3.$
\end{proof}

\begin{lemma}\label{3tight}
For $q\geq 3,$ there exists a subset $U_2(q)$ of the vertices of $H(2,q)$ of $|U_2(q)|=3q$ so that the VC-dimension of $(U_2(q),n(U_2(q)))$ is strictly less than three.
\end{lemma}

\begin{proof}
We will now construct a subset $U_2(q)$ with size $3q$ that has VC-dimension $<3.$ Call the rows $R_j$ where $j$ ranges from 1 to $q.$ For $j=1, \dots q,$ select the points from columns $j, j+1,$ and $j+2$ modulo $q.$ Every row and every column has exactly 3 points from $U_2(q),$ for a total of $3q$ points. By the contrapositive of \ref{lemma23}, this cannot have VC-dimension 3, and the tightness is proved.
\end{proof}
\subsection{The lower bound}
We now move on to the case where $|U|\geq 3q+1.$ The goal will be to find a line $L$ with four points on it, $x,y,z,$ and $u_3$ such that the first three of those points also lie on lines $L_x, L_y,$ and $L_z$, respectively, each with at least one other point from $U$ on them, called $u_x, u_y,$ and $u_z,$ respectively. If $L$ is vertical, call this configuration of seven points a {\it vertical fist}, and if $L$ is horizontal, we call this configuration a {\it horizontal fist}. We now prove that the existence of a fist and one more special point $u_0\in U$ guarantees that $(U,n(U))$ has VC-dimension 3.

\begin{lemma}\label{fist}
For $q\geq 4,$ any subset $U$ of the vertices of $H(2,q)$ with $|U|\geq 3q+1,$ the VC-dimension of $(U,n(U))$ is three.
\end{lemma}
\begin{proof}
Given a vertical fist defined as above, we will select a shattering set $W:=\{x,y,z\}.$ Recall that we need to find points $u\in U$ to give us the eight possible subsets $S$ (elements of the power set of $W$) when we intersect $W$ with the various choices of $n(u).$ Pick any point from $U$ that is not on any of $L, L_x, L_y,$ or $L_z,$ to be $u_0.$ Notice that $u_0$ is not adjacent to any vertex from $W,$ giving us the empty set. Next, we see that $u_x$ is adjacent to $x,$ but not any other vertex in $W.$ The same relationship holds for $u_y$ and $u_z$ with respect to $y$ and $z.$ So these points give us the singleton subsets of $W.$ For the two-element subsets of $W,$ we select the neighborhood of the third element. Specifically, notice that $n(x)\cap W = \{y,z\}, n(y)\cap W = \{x,z\},$ and $n(z)\cap W = \{x,y\}.$ Finally, any point from $L\setminus W$ will be adjacent to all three elements of $W.$ We often call this point $u_3.$

\begin{figure}
\centering
\includegraphics[scale=1.2]{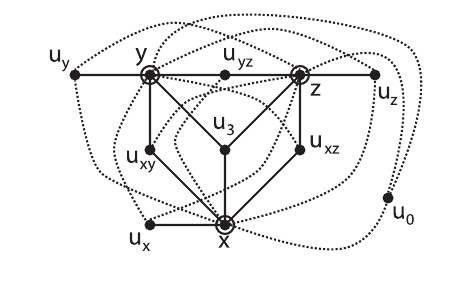}
\caption{The vertices from the shattering set $W=\{x,y,z\}$ are circled. When two vertices are adjacent, an edge is indicated with a solid line. When two vertices may not be adjacent, this is indicated with a dotted curve. Note that as long as the edge relations are allowed, some vertices in the figure may be identified with one another. So, for example, $x$ cannot be $y,$ but it could be $u_{yz}.$ }
\label{figure2}
\end{figure}

We now proceed by induction on $q.$ For the base case, set $q=4,$ so that $|U|\geq 3q+1 = 13$ and verify by hand that there must be a fist. To see this, notice that by the pigeonhole principle, there must be at least one column $L$ with four points in it, and of the remaining twelve possible points, at least nine must be in $U.$ We now split into two cases based on the locations of points in $U\setminus L.$ Since each of the four rows has three possible locations for points outside of $L$, either each row has at least one point from $U\setminus L,$ or there is a row in $U \setminus L$ with no points.

In the first case, three of the populated rows will each contribute one of $u_x, u_y,$ and $u_z,$ to a vertical fist configuration, each lying on the row alongside $x,y,$ and $z,$ respectively, chosen from $L.$ The fourth point on $L$ will be $u_3,$ and $u_0$ will be another point on that row.

In the second case, call one of the full rows $R,$ and we will find a horizontal fist. Now select three of the points in $R$ to be $x,y,$ and $z,$ respectively, and let the fourth point in $R$ be $u_3.$ Now, as each of these points has another point from $U$ in its column, we select each of these to be $u_x,u_y,$ and $u_z,$ and $u_0,$ respectively.

Since either case leads to a fist and a valid choice for $u_0,$ we have VC-dimension 3 for any subset $U$ of the vertices of $H(2,4),$ of size $\geq 13 = 3(4)+1.$ Next we address the induction step. Suppose that for all $4\leq p<q,$ we have that any subset of $H(2,p)$ of size at least $3p+1$ has VC-dimension 3. Now we look at any subset $U$ of size exactly $3q+1$ in $H(2,q).$ If $U$ had more points, we can throw them out until we have exactly $3q+1$ without increasing the VC-dimension. If we can find a row and column who have three or fewer points total in their union, then we remove that row and column, and are left with a subset of $H(2,q-1)$ with size at least $3(q-1)+1,$ and are therefore done by induction.

If not, then the union of every pair of row and column in $H(2,q)$ must have at least four points in it. Now, consider $L$ the line with the most points on it. Without loss of generality, suppose $L$ is a column with $k$ points, where, by the pigeonhole principle, we have that $k\geq 4.$ If at least four of the rows with points on $L$ have other points from $U,$ then we can choose three of the rows with points on $L$ and other points from $U$ to be called $R_x, R_y, R_z,$ and $R_0.$ We set $x= R_x\cap L,$ and let $u_x$ be another point from $U$ on $R_x.$ We do the same for $u_y, u_z,$ and $u_0,$ and see that we have a vertical fist.

However, if only three of the rows are populated by other points, then we select $R_x, R_y,$ and $R_z,$ as before, but have two simple cases to address. If there is some point off of $L\cup R_x \cup R_y \cup R_z,$ we call it $u_0$ and pick another point from $L$ to be $u_3.$ If not, then because we have exactly $3q+1$ points, we must have that $R_x, R_y,$ and $R_z$ are full rows, and we can choose one of them to be $R,$ and find a horizontal fist as we did in the $q=4$ case.

So suppose that there are at least $k-2$ rows whose only points are on $L.$ Because every pair of row and column must have a union with at least four points on it, every column must have at least 3 points on it. Now we count the points off of $L.$ Let $U'$ denote $U\setminus L.$ Since there are $k$ points on $L$, we must have that $|U'|=3q+1-k.$ However, recall that all $q-1$ columns (not counting $L$) must have at least 3 points. So we also have that $|U'|\geq 3(q-1) =3q-3.$ Comparing these bounds, we get
\[3q+1-k=|U'|\geq 3q-3\Rightarrow k\leq 4.\]
So we must have that $k=4,$ and therefore every column except $L$ has exactly three points from $U$. So, ignoring the at least two rows that have exactly one point from $L,$ there are at most $q-2$ rows with a total of $3q-3$ points from $U'.$ Therefore, by the pigeonhole principle, there must be a row with at least
\[\left\lceil\frac{3q-3}{q-2}\right\rceil = 4\]
points from $U'.$ Call this row $R.$ Since $q>4,$ there must be at least four columns excluding $L.$ Recall that every column in $U'$ has exactly three points, so pick four points on $R$ to be $x,y,z,$ and $u_3.$ Notice that each of $x,y,$ and $z$ have another point on their column, so we have a vertical fist. Finally, let $u_0$ be another point from the column with $u_3,$ and we are done.
\end{proof}

\section{Sharpness examples for $d>2$}

\subsection{Proof of Proposition \ref{sharp3}}
We now construct a set $U_3(q)$ of the vertices of $H(3,q)$ of size $|U_3(q)| = 5q^2/4$ with VC-dimension 1 for even $q$. The basic idea is to consider $H(3,q)$ as $q$ copies of $H(2,q)$ stacked on top of one another. Each pair of copies of $H(2,q)$ will be called a `layer,' so there will be a total of $q/2$ layers.

The subset in the bottom copy of $H(2,q)$ will be the set $U_1(q)$, described in the proof of Proposition \ref{prop22}. Then for the second copy of $H(2,q),$ we add one point above each of the squares in the first copy. For the next layer, we just take the first layers and shift everything to the side by two points, wrapping back around if things slide off. Each layer will have $5q/2$ points, and we can do this for $q/2$ layers, for a total of $5q^2/4$ points.

We are taking five points from the $2\times 2\times 2$ cubes that form the diagonal of the first layers, then shifting this whole arrangement to the side for each subsequent pair of layers, so that no cube is ever in a rectilinear line with any other cube. The subgraph given by the vertices in this construction will consist of $q^2/4$ connected components, consisting of five vertices each. Each component is isomorphic to $G(V,E)$ given by
\[V:=\{a,b,c,d,e\} \text{ and } E:=\{\{a,b\},\{b,c\},\{c,d\},\{d,a\},\{d,e\}\}\]

Any pair of vertices chosen from distinct components could not ever both be adjacent to the same vertex, so there would be no candidate for $u_{xy}.$ Therefore, if there is a shattering set of size two, both vertices must come from the same component. One can exhaustively check that no pair of vertices from any of the components can shatter. To see this, notice the two vertices cannot both come from the first copy of $H(2,q),$ by construction of $U_1(q).$ So one of the shattering vertices must be $e.$ Since there must be a vertex adjacent to both shattering vertices, $d$ cannot be the other shattering vertex. The vertices $c$ and $a$ would not work, as $e$ would not have a vertex that it is adjacent to while the other shatter point is not. Finally, $b$ cannot be the other shattering vertex as $b$ and $e$ have no neighbors in common.

\begin{figure}
\centering
\includegraphics[scale=1.2]{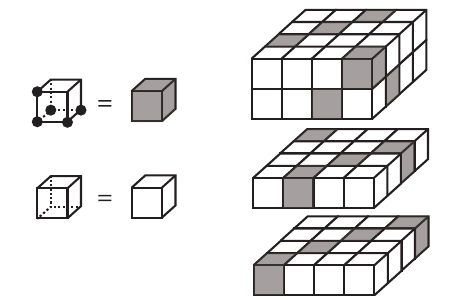}
\caption{Each $2\times 2\times 2$ subcube will either have five points (shaded) or zero (unshaded). Every layer will be a pair of copies of $H(2,q)$ will have $q/2$ of these subcubes, arranged diagonally. This pattern will shift with subsequent layer, so that shaded cubes are never in line with one another.}
\label{figure3}
\end{figure}

\subsection{Proof of Proposition \ref{sharp4+}}
For appropriate $d$ and $q,$ we will identify the vertices of $H(d,q)$ with $\mathbb Z_q^d,$ so coordinatewise addition and subtraction will be modulo $q.$ Define the set $U_d'(q)\subset H(q,d)$ by
\[U_d'(q):=\left\{(x_1, x_2, \dots, x_d): x_d = \sum_{j=1}^{d-1} x_j\right\}.\]
We now verify that there are no two vertices that are adjacent by way of contradiction. Suppose that $\vec x=(x_1, x_2, \dots, x_d)$ and $\vec y = (y_1, y_2, \dots, y_d)$ are adjacent. That means that they differ in exactly one coordinate. First, suppose they differ in their final coordinate. Then we would have
\[x_d = \sum_{j=1}^{d-1} x_j = \sum_{j=1}^{d-1} y_j = y_d.\]
But then $x_d$ would be equal to $y_d,$ which is a contradiction. Next, suppose they differ in exactly one coordinate from the first $d-1.$ Without loss of generality, suppose it is the first coordinate. Then we would have
\[x_1 = x_d - \sum_{j=2}^{d-1} x_j = y_d - \sum_{j=1}^{d-1} y_j = y_1,\]
which leads to a similar contradiction. Since $U_d'(q)$ has no adjacent vertices, the VC-dimension of $(U_d'(q),n(U_d'(q)))$ is 1.

\subsection{Proof of Proposition \ref{3sharp4+}}
We now prove a modified version of Lemma \ref{lemma23} for $d>2$. To do this, we define another point configuration that will help us to characterize subsets of vertices of $H(d,q)$ of VC-dimension 3.
\subsubsection{Rectangles}
We call  a set of four points in a plane that lie on the four intersection points of two pairs of parallel lines a {\it rectangle}. For example, the following set of points form a rectangle in $H(4,9)$:
\[\{(1,1,3,8),(1,4,3,8),(1,1,6,8),(1,4,6,8)\}.\]
\begin{lemma}\label{rectangles}
    If $U$ is a subset of the vertices of $H(d,q)$ so that the VC-dimension of $(U,n(U))$ is 3, then either $U$ has a line with at least 4 points on it, or it has a rectangle.
\end{lemma}
\begin{proof}
    Similar to the proof of Lemma \ref{lemma23}, we will begin by assuming that there is a shattering set $W\subseteq U$ with size 3. If $U$ has either four points on a line or a rectangle, we are done, so we proceed supposing that $U$ has neither of these.
    
    We now focus on the distinct points $x,y,$ and $z$ which together comprise $W.$ Because $W$ can be shattered, there must be points $u_{xy},u_3\in U$ such that $u_{xy},u_3\in n(x)\cap n(y).$ Note that $x\notin n(x)$ and $y\notin n(y),$ so neither of $x$ or $y$ can be either $u_{xy}$ or $u_3.$ Moreover, $u_{xy}\notin n(z),$ but $u_3\in n(z),$ so we have $u_{xy}\neq u_3.$ Combining these facts, we see that the set
    \[F:=\{x,u_{xy},y,u_3\}\subseteq U\]
    consists of four distinct vertices. Now we split into two cases: either $x$ and $y$ live on the same line, or they do not.
    
    In the first case, the fact that $x$ and $y$ are both adjacent to $u_{xy}$ means that $u_{xy}$ must be on the same line as both $x$ and $y,$ as it would otherwise only be adjacent to at most one of $x$ and $y$. However, the same holds for $u_3,$ and since $F$ consists of four distinct points, this means that we would have four points on a line. Since we have assumed we do not have four points on any line, the first case cannot occur, and $x$ and $y$ cannot be on the same line.
    
    To address the second case, where $x$ and $y$ do not lie on the same line, we first consider that $u_{xy}$ must be adjacent to both. So there must be a line $L_1$ through $x$ that intersects a line $L_2$ through $y,$ with $u_{xy}$ lying on their intersection. From this, we see that $x$ and $y$ can differ in at most two coordinates. Therefore, we have that $x,y,$ and $u_{xy}$ all live in some plane $P.$ Similarly, $u_3$ must be adjacent to both $x$ and $y,$ so there must be a line $L_3$ through $x$ that intersects a line $L_4$ that goes through $y$, and their intersection must be $u_3.$ But since $x,y,$ and $u_{xy}$ all live in the plane $P,$ and $u_3$ only differs from $x$ and $y$ in one coordinate each, we see that $u_3$ must also live in $P.$ Therefore the set $F$ would form a rectangle, which contradicts our assumption that there was no rectangle in $U.$
\end{proof}
\subsubsection{Constructing a set with no rectangles}
We now demonstrate the sharpness of Corollary \ref{phpCor}, as claimed by Proposition \ref{3sharp4+}, by the following construction, which is a variant of the construction in Proposition \ref{sharp4+}. We will see that this subset of $H(d,q)$ will have no line with four points, nor any rectangles, and therefore has VC-dimension strictly less than 3. However, as an arithmetic artifact of the proof, we will need $q\geq 6,$ rather than the usual $q\geq 4.$
\begin{lemma}\label{3diagonals}
    Given natural numbers $d\geq 3$ and $q\geq 6,$ there exists a subset $U$ of the vertices of $H(d,q)$ of size $|U|=3q^{d-1}$ so that $(U,n(U))$ has VC-dimension 2. 
\end{lemma}
\begin{proof}
    For appropriate $d$ and $q,$ define the set $U_d''(q)\subset H(q,d)$ by
\[U_d''(q):=\left\{(x_1, x_2, \dots, x_d): x_d = \epsilon + \sum_{j=1}^{d-1} x_j, \epsilon\in\{-1,0,2\}\right\}.\]
To visualize this set, we focus the ``two-dimensional slices," or copies of $H(2,q)$ that are subgraphs of $H(d,q)$ where only two coordinates vary. Fix any such slice and notice that the points of $U_d''(q)$ there will be two adjacent diagonal arrangements of points, one diagonal arrangement of spaces, then one more diagonal arrangement of points. See Figure \ref{figure5}.
\begin{figure}
\centering
\includegraphics[scale=1.2]{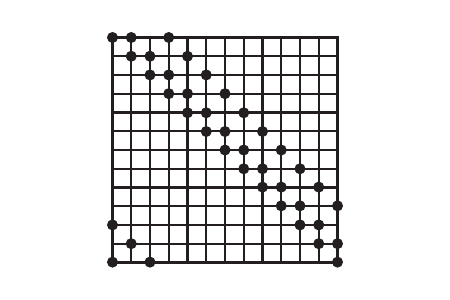}
\caption{This is a two-dimensional slice of $U_d''(12).$}
\label{figure5}
\end{figure}
We first show that this set has no line with four points. To see this, we split into two cases. Assume that either the line consists of points that are constant except in the final coordinate, or one of the first $d-1$ coordinates. We start with the case that the points vary in the final coordinate. Then all such points are of the form $\vec x = (x_1, \dots, x_d),$ where
\[x_d = \epsilon + \sum_{j=1}^{d-1}x_j,\]
where $\epsilon \in \{-1,0,2\}.$ Since the first $d-1$ coordinates are constant, $x_d$ can only take three possible values, depending only on choice of $\epsilon,$ so there can be at most three points on any line varying in only the final coordinate. Next, we deal with the case that the points on the line vary in one of the first $d-1$ coordinates. Without loss of generality, we can assume that they vary in only the first coordinate. But then $x_2,\dots, x_d$ are fixed, and all such points must satisfy
\[x_1 = x_d - \epsilon - \sum_{j=2}^{d-1}x_j,\]
giving us again only at most three choices of points, depending on $\epsilon.$
We now show that there is no rectangle in $U_d''(q).$ This fact will follow from similar, though more involved, case work. A rectangle will live in a plane, meaning that there will be exactly two coordinates varying between the four points. If there is a rectangle, it will consist of the points $\vec x, \vec y, \vec z,$ and $\vec w,$ with each point adjacent to the points it was listed next to, and with $\vec w$ adjacent to $\vec x.$ Moreover, suppose that we define $\epsilon_x$ by
\[x_d = \epsilon_x+\sum_{j=1}^{d-1}x_j.\]
Similarly, define $\epsilon_y, \epsilon_z,$ and $\epsilon_w.$ As before, we will either have that the final coordinate varies, or it does not. This gives us two cases, but we will only give the full details for the first case, as the second case follows by a similar argument.

In the first case, we will assume without loss of generality that the points agree an all coordinates except the first and last. That is, for $j=2,\dots, d-1,$ we have $x_j=y_j=z_j=w_j.$ Again without loss of generality, we assume that
\begin{equation}\label{eqs}
x_1=y_1, y_d=z_d, z_1=w_1, \text{ and } x_d=w_d.
\end{equation}
Therefore, by \eqref{eqs} and the fact that the middle $d-2$ coordinates are the same, we have that
\[x_d\neq y_d=\epsilon_y+\sum_{j=1}^{d-1}y_j = \epsilon_y+\sum_{j-1}^{d-1}x_j = \epsilon_y+(x_d-\epsilon_x).\]
From this we gather that $\epsilon_x\neq \epsilon_y.$ Arguing similarly, we get $\epsilon_z\neq \epsilon_w.$ Next, we have
\begin{align*}
    y_1\neq z_1 &= z_d-\epsilon_z-\sum_{j=2}^{d-1}z_j\\
    &= y_d-\epsilon_z-\sum_{j=2}^{d-1}y_j\\
    &=\left(\epsilon_y+\sum_{j=1}^{d-1}y_j\right)-\epsilon_z-\sum_{j=2}^{d-1}y_j\\
    &= \epsilon_y+y_1-\epsilon_z,
\end{align*}
which tells us that $\epsilon_y\neq \epsilon_z.$ Again, arguing similarly we see that $\epsilon_w\neq \epsilon_x.$ By the restrictions above, and the fact that the set $\{\epsilon_x, \epsilon_y, \epsilon_z, \epsilon_w\}\subseteq \{-1,0,2\}$, which has only three distinct elements, we see that we must have either $\epsilon_x=\epsilon_z$ or $\epsilon_y=\epsilon_w.$ Without loss of generality, suppose $\epsilon_x=\epsilon_z.$ Plugging this back in to what we know about $\vec x$ and $\vec z,$ we get
\[x_d - \sum_{j=1}^{d-1}x_j = \epsilon_x = \epsilon_z = z_d - \sum_{j=1}^{d-1}z_j.\]
Separating out the first term and recalling that the middle $d-2$ coordinates agree, we get
\[x_d - x_1-\sum_{j=2}^{d-1}x_j = \epsilon_x = \epsilon_z = z_d - z_1-\sum_{j=2}^{d-1}z_j= z_d - z_1-\sum_{j=2}^{d-1}x_j.\]
This tells us that
\[x_d - x_1 = z_d - z_1.\]
Appealing to \eqref{eqs}, this gives us
\[w_d-y_1 = y_d-w_1\Rightarrow w_d+w_1 = y_d+y_1.\]
Plugging in the definitions of $w_d$ and $y_d,$ we get
\[\left(\epsilon_w+\sum_{j=1}^{d-1}w_j\right)+w_1 = \left(\epsilon_y+\sum_{j=1}^{d-1}y_j\right)+y_1\]
Again using the fact that the middle $d-2$ coordinates agree, and subtracting their sum from both sides, gives us
\[\epsilon_w+2w_1 = \epsilon_y+2y_1 \Rightarrow \epsilon_w-\epsilon_y = 2(y_1-w_1).\]
Now, if $\epsilon_w=\epsilon_y,$ then we would have that $y_1=w_1.$ By \eqref{eqs}, this would lead to $x_1=y_1=z_1=w_1,$ which, as the points also share the middle $d-1$ coordinates, contradicts the fact that there are not four points on any line. So we proceed with the assumption that $\epsilon_w\neq \epsilon_y.$ This means that as sets,
\begin{equation}\label{setEq}
\{\epsilon_x,\epsilon_y,\epsilon_w\} = \{-1,0,2\}.
\end{equation}
To finish up the first case, we appeal to \eqref{eqs} twice, following similar reasoning to the above calculations, and compute
\[x_d+z_d = w_d+y_d\]
\[\left(\epsilon_x+\sum_{j=1}^{d-1}x_j\right)+\left(\epsilon_z+\sum_{j=1}^{d-1}z_j\right)=\left(\epsilon_w+\sum_{j=1}^{d-1}w_j\right)+\left(\epsilon_y+\sum_{j=1}^{d-1}y_j\right)\]
\[\epsilon_x+x_1+\epsilon_z+z_1=\epsilon_w+w_1+\epsilon_y+y_1\]
\begin{equation}\label{sumEq}
\epsilon_x+\epsilon_z=\epsilon_y+\epsilon_w.
\end{equation}
This is where our lower bound $q\geq 6$ becomes necessary, as the following argument fails for $q \leq 5.$ We now show that \eqref{sumEq} is not compatible with the possible values given by \eqref{setEq} and the fact that $\epsilon_x=\epsilon_z.$ To see this, one need only check each possible value of $\epsilon_x,$ because $\epsilon_x= \epsilon_z$ and by \eqref{setEq}, this choice determines the value of the sum $\epsilon_y+\epsilon_w.$ Specifically,
\[\epsilon_x+\epsilon_z = 0 + 0\neq 1= (-1)+2= \epsilon_y+\epsilon_w,\]
\[\epsilon_x+\epsilon_z = (-1)+(-1) \neq 2 = 0+ 2 \epsilon_y+\epsilon_w,\]
\[\epsilon_x+\epsilon_z = 2 + 2\neq -1= (-1) + 0\epsilon_y+\epsilon_w.\]

With the first rectangle case dealt with, we indicate how the second case is handled, but do not go into full detail, as the proof is quite similar. In the second case, we can assume without loss of generality that only the first two coordinates vary. That is, we will have $x_d=y_d=z_d=w_d.$ This gives us
\[\left(\epsilon_x+\sum_{j=1}^{d-1}x_j\right)=\left(\epsilon_z+\sum_{j=1}^{d-1}z_j\right)=\left(\epsilon_w+\sum_{j=1}^{d-1}w_j\right)=\left(\epsilon_y+\sum_{j=1}^{d-1}y_j\right).\]
Since only the first two coordinates vary, this implies
\[\epsilon_x+x_1+x_2=\epsilon_z+z_1+z_2=\epsilon_w+w_1+w_2=\epsilon_y+y_1+y_2.\]
We use this along with a version of \eqref{eqs} in terms of first and second coordinates to say
\[x_2\neq y_2 = \epsilon_x +x_1+x_2- \epsilon_y -y_1 = \epsilon_x-\epsilon_y +x_2 \Rightarrow \epsilon_x\neq \epsilon_y,\]
and a similar argument would follow.

This set has no line with four points, nor any rectangle, which means that $(U_d''(q),n(U_d''(q)))$ has VC-dimension strictly less than 3 by Lemma \ref{rectangles}.
\end{proof}

\section{Other distances}

While this paper has focused on the most common case of $H(d,q,t)$ which is $t=1,$ we conclude with some preliminary results for $t\neq 1.$ In particular, we record some trivial observations. We then get a sharp result in the case that $t=d=2,$ and use this to derive nontrivial results for $t=2$ when $d\geq3.$ However, at the time of this writing, we know little else about other values of $t\neq 1.$

\subsection{Trivial cases}
First off, if $d=q=1,$ then the VC-dimension is 0 for any $t.$ We also note that if $t>d,$ there are no edges, so the VC-dimension is 0. Next, for any $d,q\in \mathbb N,$ with $q\geq 2,$ we note that if $t=0$, the VC-dimension is clearly 1, as the graph consists of $q^d$ loops, or vertices connected only to themselves. So in the notation above, any point can be $x$ and $u_x,$ and any other point is $u_0.$

\subsection{Proof of Theorem \ref{thm222}}

Theorem \ref{thm222} will follow by combining Propositions \ref{prop2q2} and \ref{2q2tight}. The case where $t=d$ is special, as it is the graph complement of $H(d,q,1).$ So for some relationships, we can just reverse the roles of certain vertices. For example, $u_0$ in the proof of Theorem \ref{thm23} could ostensibly serve as $u_3$ in a result on $H(2,q,2)$. However, as there are no loops in $H(d,q,t)$ for $t>0,$ we cannot rely on just reversing some of the relationships we used before. For example, in the proof of Proposition \ref{prop22}, we let $x=u_y,$ but we will not be able to use $x=u_x$ in a result on $H(2,q,2).$ So we use ideas from \cite{PSTT} to prove the following result.

\begin{proposition}\label{prop2q2}
    Given a natural number $q\geq 3,$ and subset $U$ of the vertices of $H(2,q,2)$ of size $|U|\geq 2q$, the VC-dimension of $(U,n(U))$ is at least 2.
\end{proposition}
As was the case with Corollary \ref{phpCor}, we can pigeonhole this result to say something about $H(d,q,2)$ for $d\geq 3.$ Specifically, we get the following.

\begin{corollary}\label{phpCor2}
    Given natural numbers $d \geq 2,$ and $q\geq 4,$ and a subset of $U$ of the vertices of $H(d,q,2)$ of size $|U|\geq 2q^{d-1},$ then the VC-dimension of $(U, n(U))$ is 2.
\end{corollary}

\subsubsection{Row-pluck and column-pluck configurations}
To prove Proposition \ref{prop2q2}, we introduce configurations similar to those used in \cite{PSTT}. If there is a row with at least two points, called a {\it principal row}, and at least two of those points, called {\it pivots}, have other points in their columns, we will call this set of four points a {\it row-pluck} configuration. Any point off of the row or columns of the pivots is called an {\it unrelated} point relative to that particular row-pluck configuration. Reversing the roles of columns and rows will give us a {\it principal column} and a {\it column-pluck} configuration.

We now state and prove a lemma showing that the existence of either one of these configurations in a subset $U$ of the vertices of $H(2,q,2)$ is enough to guarantee that the VC-dimension of $(U,n(U))$ is at least 2.

\begin{lemma}\label{lemmaRich}
    Given a natural number $q\geq 3,$ and subset $U$ of the vertices of $H(2,q,2)$ with a row-pluck (or column-pluck) configuration, as well as an unrelated point relative to that row-pluck (or column-pluck) configuration, the VC-dimension of $(U,n(U))$ is at least 2.
\end{lemma}
\begin{proof}
    Without loss of generality, suppose that there is a row-pluck configuration. We will use the same naming convention for the shattering set $W$ that we have been using. That is, we will find points $x$ and $y$ in $U$ so that $W=\{x,y\}$ can be shattered by $n(u_0), n(u_x), n(u_y),$ and $n(u_{xy}),$ for some choices of $u_0, u_x, u_y, u_{xy}\in U.$
    
    We will refer to the two points on the same row, whose columns are also populated by other points, by the names $x$ and $y.$ The point on the same column as $x$ will be called $u_y.$ We see that $u_y$ is adjacent to $y$ but not to $x,$ as it is Hamming distance 2 to $y$ and Hamming distance 1 to $x.$ Similarly, the point on the same column as $y$ will be called $u_x.$ The unrelated point will serve as $u_{xy},$ and $x$ can be $u_0.$
\end{proof}

\subsubsection{Proof of Proposition \ref{prop2q2}}
This proof is similar to the proof of Proposition \ref{prop22} from \cite{PSTT}. We use induction on $q$ to show that there must be a row-pluck or column-pluck configuration, as well as a suitable unrelated point. For the base case, when $q=3,$ we see that $U$ has size at least $2q = 6.$ Let $R$ be a row with the most points. By the pigeonhole principle, $R$ must have at least two points. If we ignore the two or three points in $R$ for the moment, then we have at three points in the other two rows. Pigeonholing again, we see that there must be another row $R'$ with at least two points. Notice that $R$ and $R'$ each have at least two of their three columns populated, meaning that by pigeonholing again, there must be a column $C$ with a point from $R$ and a point from $C.$ The points $R\cap C$ and $R'\cap C$ will serve as pivots for a column-pluck with principal column $C$. Now, if there is a point unrelated to this column-pluck, we are done by appealing to Lemma \ref{lemmaRich}. If there is not a point unrelated to this column-pluck, that means that at least one of $R$ or $R'$ is full. Since $R$ had the maximum number of points, it is full. In this case, we find a row-pluck with principal row $R',$ with one pivot as $R'\cap C,$ and one of the other columns $C'$ providing the other pivot as $R'\cap C'.$ Finally, the point unrelated to this row-pluck will be $R\setminus(C\cup C'),$ and we are again done by Lemma \ref{lemmaRich}. This completes the base case. 

For the induction hypothesis, we will assume that for any $3\leq p < q,$ the claimed result holds for any subset of the vertices of $H(2,p,2)$ of size $2p.$ We now show that this implies that the claim will also hold for $H(2,q,2).$ Call the row with the most points $R$. Without loss of generality, suppose that it has as many or more points than the column with the most points. Let the number of points on $R$ be called $k,$ which is by definition at most $q.$ By the pigeonhole principle, we know that $k\geq 2.$ Now, either the row $R$ has points in at least two columns $C$ and $C'$ that are populated in other rows or it does not. In the case that it does, $R$ will serve as the principal row and the two points $R\cap C$ and $R\cap C'$ will serve as pivots. Now, either there is a point unrelated to this row-pluck, or there is not. If there is, we are done. If not, we see that all points from $U$ live in $R\cup C \cup C'.$ Here we will find a column-pluck and a point unrelated to it. To see that this must happen, notice that $R$ has $k$ points, meaning that $(C\cup C')\setminus R$ has at least $2q-k\geq q$ points in a total of $q-1$ rows besides $R$. By the pigeonhole principle, that means at least that at least one of the rows $R'$ has points from both $C$ and $C'.$ Without loss of generality, suppose $C$ has as many or more points as $C'.$ We then choose our column-pluck so it has pivots $R\cap C'$ and $R'\cap C'.$ Next we show that and some point from $C\setminus(R\cup R')$ must exist, so it can serve as the unrelated point. To see that this point must exist, we recall that there are at least $q$ points in $(C\cup C')\setminus R,$ at least half of which are on $C,$ and only one of which is on $R'.$ Therefore, as $q\geq 3$ there are at least $\lceil q/2 \rceil - 1 > 1$ potential choices for the unrelated point.

If $R$ does not have two points with otherwise populated columns, then we will reduce to a smaller case. This reduction follows by selecting a column $C$ with a point in $R$ that has the maximal number of points. Let $\ell$ denote the number of points in $C.$ By definition, $\ell\leq k.$ Now remove $R$ and $C,$ and are left with a smaller set of vertices. In particular, as $k\geq 2,$ we are now considering a set of at least $2q-k-(\ell-1)\geq 2(q-1)$ points in $\mathbb Z_{q-1}\times \mathbb Z_{q-1}.$ So we can appeal to the induction hypothesis for $p=q-1,$ and we are done.

\subsubsection{Tightness}
The following result shows that we have the correct exponent and constant multiple on $q$ in the lower bound given in Proposition \ref{prop2q2}.

\begin{proposition}\label{2q2tight}
    For $q\geq 3,$ there exists a subset $U_*(q)$ of the vertices of $H(2,q,2)$ of $|U_*(q)|=2q-1$ so that the VC-dimension of $(U_*(q),n(U_*(q)))$ is strictly less than two.
\end{proposition}
\begin{proof}
    We will define $U_*(q)$ to be the subset of the vertices of $H(2,q,2)$ with at least one coordinate being 0. There are $2q-1$ such points. Now, notice that $(0,0)$ is not adjacent to any other points in the set, so it cannot serve as $x$ or $y.$ So both $x$ and $y$ have exactly one nonzero coordinate. Either they are on the same line or they are not. In the first case, without loss of generality, suppose $x=(a,0)$ and $y=(b,0).$ But then $n(x)=n(y),$ so there is no choice for $u_x.$ In the second case, without loss of generality, suppose $x=(a,0)$ and $y=(0,b).$ But then there is no point adjacent to both of them, so there is no possible choice for $u_{xy}.$
\end{proof}

\end{document}